\documentclass[a4paper,draft,oneside,12pt]{article}

\usepackage{ amsthm,amsmath,amsfonts,amssymb}
\usepackage{epsfig}
\usepackage{colordvi,helvet,multicol}
\usepackage{epic}
\usepackage{mathptmx}
\usepackage{times}
\usepackage[T1]{fontenc}
\usepackage{verbatim}
\usepackage{graphics}
\usepackage{makeidx}
\usepackage{fancyhdr}
\makeindex
\makeatletter
\long\def\@caption#1[#2]#3{%        Abbildungstext auf Fussnotengroesse
  \par
  \addcontentsline{\csname ext@#1\endcsname}{#1}%
    {\protect\numberline{\csname the#1\endcsname}{\ignorespaces #2}}%
  \begingroup
    \@parboxrestore
    \if@minipage
      \@setminipage
    \fi
   \normalsize
    \footnotesize
    \@makecaption{\csname fnum@#1\endcsname}{\ignorespaces #3}\par
  \endgroup}
\makeatother
 \theoremstyle{plain}
\newcommand\DD{{\Bbb D}}
 
\newcommand{\ee}{\end{equation}}

\newtheorem{definition}{Definition}[section]
\newtheorem{theorem}[definition]{Theorem}
\newtheorem{lemma}[definition]{Lemma}
\newtheorem{remark}[definition]{Remark}
\newtheorem{proposition}[definition]{Proposition}
\newtheorem{corollary}[definition]{Corollary}

\newtheorem{pr1}[definition]{}

\newcommand\CC{{\Bbb C}}

\newcommand\TT{{\Bbb T}}

\begin{document}

\title{ On  Character Amenability of   Banach Algebras
 %\footnote{This article is motivated by my PhD thesis at the Technical university of Munich. }
}
   \author{Ahmadreza Azimifard
   %\\azimifard@yahoo.com
   \date{}
   }

\maketitle

\begin{abstract}
            Associated to a nonzero homomorphism $\varphi$ of a Banach algebra  $A$,
            we regard  special functionals, say $m_\varphi$,  on certain subspaces of $A^\ast$
            which provide  equivalent statements  to the
            existence of a  bounded right approximate identity in the corresponding maximal ideal in $A$.
            For instance, applying a fixed point theorem yields an equivalent statement to the existence of a $m_\varphi$
            on $A^\ast$;  and,  in addition  we  expatiate
            the case that if a functional $m_\varphi$ is unique, then $m_\varphi$
             belongs to the topological centre of   the bidual algebra $A^{\ast\ast}$.
            An example of a function algebra, surprisingly,  contradicts  a conjecture that  a
         Banach algebra $A$ is amenable if $A$ is  $\varphi$-amenable in every character $\varphi$ and if  functionals
         $m_\varphi$ associated to the characters $\varphi$ are uniformly bounded.
            Aforementioned  are also elaborated   on  the direct sum  of two given Banach algebras.
%\footnote{On Oct.10, 2007, this paper is submitted to Proc. Edinb. Math. Soc.}
\end{abstract}

\begin{tabbing}
\end{tabbing}


% \footnotesize{
%\begin{tabular}{lrl}
% {\bf{  2000 AMS  Subject Classification.}}& \multicolumn{2}{l}{43A20.}\\
%  \hspace{3.cm}  {\bf{  Keywords.}} &{\emph{ Banach algebra, topological centre, amenability.}}
% \end{tabular}
%}

\section{Introduction}

            A. T. M.  Lau introduced and studied
            the left  amenability of certain Banach    algebras
            \cite{Lau83}. In the  special cases
            he pointed out that for a locally compact group $G$, the
            left amenability of $L^1(G)$, the  group algebra,  and the left amenability of $M(G)$, the measure algebra, are equivalent
             to
            the amenability of $G$. The latter, due to B. Johnson \cite{Johnson},
            is  also equivalent to the amenability of $L^1(G)$  as well as  the existence of
              a bounded right approximate identity in the maximal ideal of
              $L^1(G)$, consisting of functions with vanishing integral.
             Hence, the left amenability of $L^1(G)$ (with respect to the identity character)
             is equivalent to the amenability of $L^1(G)$.
            We may refer the reader e.g.  to \cite{Dales, Greenleaf, Pat88} and  \cite{ Runde} for extensive treatments
            of various notions of amenability.
%%%%%%%%%%%%%%%%%%%%%%%%%%%%%%%%%%%%%%%%%%%%%%%%%%%%%%%%%%%%%%%%%%%%%%%%%%%%%%%%%%%%%%%%%%%%%%%%%%%%%%%%%%%%%%%%%%%%%%%%

             One may now
            pose a natural question of whether the  above statements
            remain valid for a Banach algebra with a nonzero homomorphism on
            it.

%%%%%%%%%%%%%%%%%%%%%%%%%%%%%%%%%%%%%%%%%%%%%%%%%%%%%%%%%%%%%%%%%%%%%%%%%%%%%%%%%%%%%%%%%%%%%%%%%%%%%%%%%%%%%%%%%%%%%%%%%

          In \cite{thesis}  we initiated and studied the notion of (nontrivial) character (left) amenability for
            a commutative hypergroup algebra, and, among other things,  we showed that a hypergroup algebra is
            character (left)
            amenable if and only if its  maximal ideal associated  to the character has a bounded approximate
            identity. Moreover, with a large family of examples of hypergroups  we revealed  that
             this kind of amenability of hypergroup algebras depends heavily on the
              asymptotic behavior of both Haar measures of hypergroups
               and characters. In particular,
               the algebra may in general fail to be even weakly amenable \cite{thesis, Ska92}.
               So, we may thence, in general, infer a negative answer to the anterior question.

               In this work we   aim  to  develop   a  notion of character amenability for Banach
               algebras, which,  in particular,  we generalize
               major results in \cite{Lau83}. The article is
               organized as follows:

               Section \ref{preliminaries} briefly presents  pertinent notions and notations  on Banach
               algebras, whereas
             Sections \ref{amenability},  \ref{sub.amenability} and
             \ref{e.acc.} contain main results and examples.
                        Let $A$  denote a   Banach algebra
                        and $Hom(A, \CC)$  the set of all nonzero homomorphisms (characters) from $A$
                         into $\CC$.
                      Suppose   $\varphi\in Hom(A, \CC)$ and $I_\varphi(A)$
           designates  the maximal ideal in $A$ determined  by $\varphi$.
           Theorems \ref{theorem.1},   \ref{T.4.1.2},  and \ref{Coro.2}
           provide  equivalent statements to the
            $\varphi$-amenability of $A$.
                        Additionally, Theorem \ref{Coro.2} also utters that
                       $A$   is $\varphi$-amenable with a bounded right approximate identity
                       if and only if
                       $I_\varphi(A)$ has a bounded right approximate identity.
           If $A$ is commutative, regular,    and  semisimple,  and $\{\varphi\}$ is a spectral set,
            then by Proposition \ref{L.7}
           every  $\varphi$-derivation on $A$ is zero.
                                The $\varphi$-left amenability of certain subspaces of $A^\ast$ with a
                                unique $\varphi$-mean is   elaborated  in
                                 Theorem \ref{T.6} and Remark \ref{R.2}.

                            In Section \ref{sub.amenability} we consider the Banach algebra $A\oplus_{\phi} B$
                            of direct
                             sum    of given Banach algebras $A$ and $B$, where $\phi \in Hom(B, \CC)$.
                             In this section,  we first show that
                             $(A\oplus_\phi B)^{\ast\ast} \cong A^{\ast\ast}\oplus_\phi B^{\ast\ast}$
                             (Proposition \ref{t.6}), and then    Theorem \ref{t.7}
                             indicates the isomorphism
                             $Z\left((A\oplus_\phi B)^{\ast\ast}\right)
                             \cong Z(A^{\ast\ast})\oplus_\phi Z(B^{\ast\ast})$ of topological centres.  Proposition \ref{1.33}
                              determines
                              $Hom(A\oplus_\phi B, \CC)\cong Hom(A, \CC)\times \{\phi\}$, and Theorem
                               \ref{1.34}  affirms  that
                              $A\oplus_\phi B$ is $\varphi\times \phi$-amenable if and only if $A$ is
                               $\varphi$-amenable.

                               Section \ref{e.acc.} concludes the
                               paper with examples of
                               Banach  function algebras whose  $\varphi$-
                               amenability depends on
                                  characters $\varphi$.

 \section{Preliminaries}\label{preliminaries}
            For a linear space $X$ and a function $f$ on $X$, the value of $f$
            at $x\in X$
             is denoted  by $f(x)$, $\langle f,x\rangle$
            or $\langle x, f\rangle $. The usual dual and bidual spaces of $X$ are specified by $X^\ast$ and $X^{\ast\ast}$,
            respectively; and, $\pi$ designates  the canonical embedding of  $X$ into
                       $X^{\ast\ast}$.

                        Let $A$ be a Banach algebra and $X$ a subspace of   $A^\ast$. For
                         $a\in A$ and $f\in X$ define
                         $a\cdot f$ and $f \cdot a$ of $A^\ast$ by
                         \[\langle f\cdot a, b\rangle=
                         \langle f, ab\rangle, \hspace{.5cm} \langle a\cdot f, b \rangle =\langle f, ba\rangle, \hspace{1.cm}(b\in A).\]
                $X$ is called  $A$-left (resp. right) invariant if $X\cdot A\subseteq X$ (resp. $A\cdot X\subseteq X$);
                it is called  $A$-invariant if it is
                 $A$-left and right invariant. For $m\in X^\ast$ and $f\in X$ define $m\cdot f\in A^\ast$ by
                 \[\langle m\cdot f,a\rangle =\langle m,f\cdot a\rangle, \hspace{1.cm} (a\in A).\]
                 When  $X=A^\ast$, $X$ is a Banach  $A$-bimodule, and $A^{\ast\ast}$ with the  following  Arens product is
                 a Banach algebra \cite{BonDun73}:
                 \[\langle n\cdot m,f\rangle=\langle n,m\cdot f\rangle, \hspace{1.cm} (m,n\in A^{\ast\ast}, f\in A^\ast).\]
                 Let $Hom(A, \CC)$ denote   the set  of all nonzero homomorphisms from  $A$ into $\CC$. For every
                                $\varphi\in Hom(A, \CC)$, we have $\pi(\varphi)\in Hom(A^{\ast\ast}, \CC)$ \cite{Civ},  and
                                 $I_\varphi(A):=\{a\in A: \varphi(a)=0\}$
                                 is a maximal ideal in $A$.  Further, if $A$ is commutative, then $Hom(A,
                                 \CC)$ is called
                                   the character space of  $A$ and   will be denoted by
                                 $\Phi(A)$. Observe that $\Phi(A)$
                                is a locally compact
                                Hausdorff space with the $w^\ast$-topology,
                                   and every maximal ideal of $A$ is given by $I_\varphi(A)$,
                                for some $\varphi\in \Phi(A)$  \cite{BonDun73}. A net $\{e_i\}_i$ in $A$ ($\|e_i\|<M$ for some $M>0$)
                                is called a  bounded right (resp. left) approximate identity
                                  if $\|ae_i-a\|\rightarrow 0$ (resp. $\|e_ia-a\|\rightarrow 0$ )
                                as $i\rightarrow \infty$,
                                 for every $a\in A$. It is called  a bounded approximate identity, if it   is
                                   a bounded left and right  approximate identity.

\begin{definition}
\emph{Let $A$ be a  Banach algebra and  $X$ a linear, closed,
$A$-left invariant subspace
                 of $A^\ast$ such that $Hom(A, \CC)\subseteq X$.  \;
                 $X$ is called   $\varphi$-left amenable if
  there exists  a  $m_\varphi\in X^{\ast}$    with the
   following properties:}
\begin{itemize}
    \item [\emph{(i)}] $m_\varphi (\varphi)=1$
    \item [\emph{(ii)}] $m_\varphi(f\cdot a)=\varphi(a)m_\varphi(f),  \hspace{.3cm}
    \emph{for every } f\in X \emph{ and }a \in A.$
\end{itemize}
 \end{definition}
Plainly  the dual space $A^\ast$ is linear, closed, and
$A$-invariant. A Banach algebra $A$ is called $\varphi$-amenable if
$A^\ast$ is $\varphi$-left amenable.
%Note that if $\varphi=1$, then part (ii) of definition above have to be considered  for all $a\in A$ such that $\pi(a)$
%is a positive functional on $A^\ast$ such that $\langle \pi(a), 1\rangle =1$.
%
%

Let
   $B(A^\ast)$ denote the space of  bounded linear operators of $A^\ast$ into $A^\ast$.
   A net of operators $\{T_i\}$ in $B(A^\ast)$ is said to be converges to $T$ in  the $w^\ast$-operator
   topology,
    denoted by $w^\ast o t$ , if and only if for every $a\in A$, $T_i a\rightarrow Ta$ in  the
    $w^\ast$-topology on $A^\ast$,  as $i\rightarrow \infty$.
 %
    %  we  mean the locally convex topology determined by the
%family of seminorms $\{\rho_{f,a}: f\in A^\ast \text{  and } a\in
%A\}$, where $\rho_{f,a}(T)=\left|T f(a)\right|$ for $T\in B(A^\ast)$.
 As  shown  in \cite{kadison},  the unit ball of $B(A^\ast)$ is compact in the $w^\ast ot$.
 \\[-1.cm]
\section{$\varphi$-Left Amenability of   Subspaces of $A^\ast$}\label{amenability}
\begin{theorem}\label{theorem.1}
\emph{Let $A$ be a Banach algebra and  $X$ a linear, closed,
$A$-left invariant subspace
 of $A^\ast$ such that $Hom(A, \CC)\subseteq X$. Suppose
$P:A^\ast\rightarrow X$ is a continuous linear map  such that $P(Hom(A, \CC))\subseteq Hom(A, \CC)$
and $P(f\cdot a)=P(f)\cdot a$, for every $f\in A^\ast$ and $a\in A$.
Then $X$ is $\varphi$-left amenable if and only if $A$ is $\varphi$-amenable.}
\end{theorem}

\begin{proof}
                               Let $X\subset A^\ast$ be $\varphi$-left amenable.
                                    Then there
                                    exists a $m_\varphi\in X^\ast$ such that $m_\varphi(\varphi)=1$
                                    and $m_\varphi(f\cdot a)=\varphi(a)m_\varphi(f)$,  for every $f\in X$
                                    and $a\in A$. Since  $P(\varphi)\cdot a=\varphi(a)P(\varphi)$ for
                                    every $a\in A$, we have $P(\varphi)=\varphi$. The continuity of $P$ implies that
                                     the functional
                                     $m'_\varphi:A^\ast\rightarrow \CC$, defined  by
                                    $m'_\varphi:=m_\varphi \circ P$,
                                    belongs     to $A^{\ast\ast}$, and
                                    $m'_\varphi(\varphi)=m_\varphi(P\varphi)=m_\varphi(\varphi)=1$.
                                    Moreover, for every $f\in A^\ast$ and $a\in A$, we may have
                                    \begin{align}\notag
                                     m'_\varphi(f\cdot a)&=m_\varphi\circ P(f\cdot a)\\
                                     \notag
                                     &=m_\varphi( P(f\cdot a))\\ \notag &=m_\varphi(P(f)\cdot a)\\
                                     \notag
                                     &=\varphi(a)m_\varphi(Pf)\\ \notag &=\varphi(a)m_\varphi\circ
                                     P(f).
                                     \end{align}
Hence, $m'_\varphi$ is a   $\varphi$-mean on $A^\ast$. The converse
of the theorem is trivial.

\end{proof}

\begin{pr1}{\bf{Examples and Remarks:}}\label{R.2}
\emph{
\begin{itemize}
    \item[(i)]
           Let $A$ be a Banach algebra with a  bounded approximate identity and  $X=AA^\ast$.
           We observe that $X$ is a  (proper) closed  submodule  of $A^\ast$ \cite{HewRos79II},
           and $Hom(A, \CC)\subset AA^\ast$.
           For a fixed $b\in A$ the map  $P_b:A^\ast\rightarrow X$ defined by $P_b(f)=b\cdot f$ is
           a continuous linear map  and  $P_b(f\cdot a)=P_b(f)\cdot a$,
           for every $a\in A.$ Then by the  previous theorem $X$ is $\varphi$-left
           amenable if and only if $A$ is
           $\varphi$-amenable.\\
           We remark  that two $\varphi$-means $m_\varphi$ and $m'_\varphi$ on $A^\ast$ are equal if and only if they are equal on
           $A^\ast A$, since
           \[\varphi(a)m_\varphi(f)=m_\varphi(f.a)=m'_\varphi(f.a)=\varphi(a)m'_\varphi(f),\hspace{.2in} \mbox{ for every }f\in A^\ast
                                                 \mbox{  and } a\in A.\]
 \item[(ii)]   Let $G$ be  a locally compact group  and $A = L^1(G)$ the group algebra.
                 Then $A^\ast\cong L^\infty(G)$ and
               $L^1(G)\ast L^\infty(G)\cong UC_r(G)$,   the algebra of bounded uniformly
                right continuous functions on $G$.
                By  part (i) and the previous theorem,  $UC_r(G)$ is $1$-amenable
                if and only if $G$ is amenable which is  already  known
                 e.g. in   \cite{Pat88}. It is worth noting that, $AP(G)$ [$WAP(G)$],
                 the Banach algebra of [weakly] almost periodic functions on $G$,  is
                  norm closed two sided translation invariant subspace of $L^\infty(G)$ and is
                 always 1-amenable even if $G$ is not amenable \cite[p.88]{Pat88}.
                  % comment for myself:
                  %UC(G)  is a $L^1(G)$-bimodule.
          \item[(iii)] Notice  that in the above theorem,   the subspace $X$ of $A$
                        need not to be an algebra.  For example, if $K$ is a locally
                        compact hypergroup, then $L^1(K)\ast L^\infty(K)\cong UC_r(K)$ may, in general, fail to be an algebra .
                        However, $UC_r(K)$ is 1-amenable if and only if $K$ is amenable \cite{Ska92}. And,
                        by part (i) and the preceding theorem,   $UC_r(K)$ is $\varphi$-amenable if and only if $L^1(K)$ is $\varphi$-amenable,
                        which is already known in   \cite{azimifard.Monath} for commutative hypergroups.
\end{itemize}
   }
 \end{pr1}

 %In the following theorem we provide an equivalent statement to the theorem above.

              \begin{theorem}\label{T.4.1.2}
              \emph{Let $A$ be a commutative Banach algebra  and
               $\varphi\in\Phi(A)$. Then
               $A$ is $\varphi$-amenable if and only if
               for every $f\in A^\ast$,
               $\lambda \varphi \in \overline{K(f)}^{ w^\ast}$
                for
                some $\lambda\in \CC$, where
                $K(f)=\{f\cdot a; \hspace{.1in}a\in A\}$ is considered with the $w^\ast$-topology.}
                         \end{theorem}

\begin{proof}
                 Let $A$ be a  $\varphi$-amenable Banach
                 algebra  and   $m_\varphi\in A^{\ast\ast}$,  a $\varphi$-mean  on $A^\ast$.
                  By Goldstein's theorem \cite{Dunford}, there exists  $\{m_{(\varphi,i)}\}_i$
                     a net in $A$ such that
                     $\pi(m_{(\varphi,i)}) \overset{w^{\ast}}\longrightarrow m_\varphi$,
                     and hence
                     $\pi(m_{(\varphi,i)})\cdot a\overset{w^\ast}{\longrightarrow} m_\varphi\cdot a$,
                     for every $a \in A$, as $i\rightarrow \infty$. Thus
 %
%\begin{align}\notag
\[\underset{i \rightarrow \infty}{\lim}\; \langle \pi(m_{(\varphi,i)})\cdot a, f\rangle
= \langle m_\varphi,f\cdot a\rangle =\varphi(a) \langle m_\varphi,f\rangle, \hspace{.2cm} \mbox{ for every }f\in A^\ast.   \]
%\end{align}
By letting  $\langle m_\varphi,f\rangle=\lambda $, we find $\lambda
\varphi \in
\overline{{K}{(f)}}^{w^\ast  }$.

               In order to prove  the converse of the theorem,  let $S=\{T_m : \hspace{.1in}T_m: f\mapsto m\cdot f,  f\in A^{\ast},m \in B'\}$,
               where $B'$ denotes  the bidual space of the unit ball $B$ in
               $A$.
Since  the  unit Ball
               in $B(A^\ast)$  is compact \cite{kadison}, the set  $\overline{S}$,  the $(w^\ast ot)$ closure of
               $S$ in $B(A^\ast)$, is compact.
%
%               $S$ is a subset of the unit ball
%               in $ B(A^\ast)$ and we let   $\bar{S} $
%               denote the $(w^\ast ot)$ closure of
%               $S$ in $B(A^\ast)$. The compactness of the  unit Ball
%               in $B(A^\ast)$ implies the compactness of  $\overline{S}$; see \cite{kadison}.
               Let  $f\in A^\ast$ and define $S_f:=\{T_m  f: \hspace{.1cm} m \in B'\}$ and
               $[{A^{\ast}}]_{(S,\varphi)}:=\{ f\in A^\ast:\hspace{.1cm}T_m f=m(\varphi) f, \text{ for every }  m\in B'\}.$
               We    observe  that
                              $\overline{S_f}^{w^\ast}\cap [{A^{\ast}}]_{(S,\varphi)}$ is nonempty for
               every $ f\in A^{\ast}$;
               in fact,  since $\lambda \varphi\in \{ T_{\pi(a)} f,\hspace{.1cm}  a
               \in B\overline{{\}}}^{w^\ast }$,   there exists
               $\{a_i\}$ a net in $B$ such that
               $T_{\pi(a_i)} f \overset{w^\ast}{\longrightarrow} \lambda \varphi$,  as
               $i\rightarrow \infty$.
               Since  $B'$ is
               $w^\ast$-compact (Banach Alaoglu's theorem),     we  may  let $m_{( f,\varphi)}$ be  a
               $w^\ast$-accumulation  point of  $\{\pi(a_i)\}_i$  in $B'$.
               Then
\begin{align}\notag
\langle m_{(f,\varphi)}\cdot f, b \rangle& =\langle m_{(f,\varphi)},f\cdot b
\rangle\\ \notag &=\underset{i\rightarrow \infty}{\lim} \langle \pi(a_i), f\cdot b\rangle
 =\lambda \langle \varphi, b\rangle, \hspace{.2cm}\text{ for every }  b\in A, \notag
\end{align}
i.e., $ T_{m_{(f,\varphi)}} f=\lambda\varphi$.  Thus
\[T_mT_{m_{(f,\varphi)}}(f)=T_m\left(T_{m_{(f,\varphi)}}f\right)
                           =\lambda T_m\varphi
                           =\lambda m\cdot \varphi
                           =\lambda m(\varphi)\varphi
                           =m(\varphi)T_{m_{(f,\varphi)}}f,\hspace{.1cm} \mbox{ for every }m\in B',\]
which implies that $T_{m_{(f,\varphi)}}\in[A^\ast]_{(S,\varphi)}$.
 For any $f\in A^\ast$,
  let $Z( f)=\left\{T_m :T_mf\in
[{A^{\ast}}]_{(S,\varphi)}\right\}$. The subsets
 $Z( f)$ of $S$ are closed and have   the finite
intersection property.
 To verify  this, suppose that   $\{ f_1,
f_2,... f_n\}$ is   any finite subset of $A^{\ast}$ and $T_{m}
\in \bigcap _{i=1}^{n-1} Z( f_i)$,  then
$T_{m} f_i \in [{A^{\ast}}]_{(S,\varphi)} $. Since
                     $\overline{S_{T_mf_n}}^{w^\ast}\cap [{A^{\ast}}]_{(S,\varphi)} \not= \emptyset $,
                      there exists a
                      $m'\in B'$ such that $T_{m'}(T_m f_n)\in
[{A^{\ast}}]_{(S,\varphi)} $,  accordingly $ T_{m'\cdot m}\in \bigcap
_{i=1}^{n} Z  ( f_i).$
 Because  of $T_m f_i\in [{A^{\ast}}]_{(S,\varphi)}, 1\leq
i\leq n-1$, we have
\begin{align}\notag
T_nT_{m'\cdot m}  f_i &=T_{n}(T_{m'}T_mf_i) \\ \notag
&=m'(\varphi)T_nT_mf_i\\ \notag
 & =m'(\varphi)n(\varphi)T_{m} f_i\\ \notag
 & =n(\varphi)T_m'T_{m} f_i\\ \notag
 &=n(\varphi)T_{m'\cdot m}f_i, \hspace{1.cm} \mbox{ for every }n\in B', 1\leq i\leq
n-1.
  \end{align}
 The compactness of $\overline{S}$ yields
  $\underset{ f\in A^{\ast}}{\bigcap} Z( f)\not=\emptyset $; so,  let  $P$ be an
element of this intersection.
There exists $\{T_{m_i}\} \subseteq\underset{ f\in A^{\ast}}{\bigcap} Z( f)$  a net of operators such that
$T_{m_i}\overset{w^\ast op}{\longrightarrow }P$, as $i\rightarrow \infty$. Then
\[\langle P(T_af), b\rangle =\underset{i\rightarrow \infty}{\lim}\langle T_{m_i}(T_af), b\rangle =
\underset{i\rightarrow \infty}{\lim}\langle T_a(T_{m_i}f), b\rangle =
\varphi(a)\underset{i\rightarrow \infty}{\lim}\langle T_{m_i}f, b\rangle =\varphi(a)\langle P(f), b\rangle ,\]
for every $f\in A^\ast$, $a\in B$ and $b\in A$.
Now let $ m_\varphi \in A^{\ast\ast}$ such that
 $m_\varphi(P(\varphi))=1$, and define $m:A^{\ast}\longrightarrow \CC$ by   $m ( f)=m_{\varphi}(P( f))$.
   Because  of  $m (\varphi)=m_\varphi(P(\varphi))=1$, and
\begin{align}\notag
m(f\cdot a)&=m(T_{a}f)\\ \notag
&=m_\varphi(P(T_{a} f))\\
\notag   &=
\varphi(a)m_\varphi(P( f))\\
\notag &=\varphi(a)m(f),\hspace{.2cm} \text{ for every } a\in B \text{
and } f\in A^\ast,
\end{align}
 $m$ is a $\varphi$-mean on $A^\ast$.
\end{proof}

Let   $X$ be  a Banach $A$-bimodule and $\varphi\in Hom(A, \CC)$. Then
$X^\ast$  in a canonical way is  a Banach $A$-bimodule. A Banach $A$-bimodule  $X$
is called  a  Banach $\varphi$-left (resp. right) $A$-bimodule if the left (resp. right) module
  multiplication is   $a\cdot x=\varphi(a)x$ $\left( \mbox{resp. }x\cdot a=\varphi(a)x\right)$,  for all $a\in A$ and $x\in X$.
A continuous linear  map   $D:A\rightarrow X^\ast$ is called a  derivation if  $D(a b)=D(a)\cdot b+a\cdot D(b)$,  for every $a, b\in A$,
and  it is called an
 inner derivation if there exists  a
$\psi\in X^{\ast}$ such that $ D(a)=a\cdot \psi-\psi\cdot a$,  for every $a\in
A$. Here $ " \cdot " $ denotes the module multiplications.
A Banach algebra $A$ is called ($\varphi$-left)  amenable if for every   Banach ($\varphi$-left) $A$-bimodule $X$,
every derivation $D:A\rightarrow X^\ast$ is inner.
%If   $\varphi\in \Phi(A)$, then  $A$ is called $\varphi$-left (right)
%amenable if $A$ is amenable whenever
% arbitrary   Banach $A$-bimodules $X$ are
%  considered with the left (right) module
%   multiplication $a.x=\varphi(a)x$ $\left(x.a=\varphi(a)x\right)$ for all $a\in A$ and $x\in X$.
   In the case that $X$ is a Banach  $\varphi$-$A$-bimodule, i.e. $a\cdot x=x\cdot a=\varphi(a)x$,
   for all $a\in A$ and $x\in X$,
    the map $D$ has   the
     form $D(ab)=\varphi(a)D(b)+\varphi(b)D(a),$ for all $a, b\in A$,
    which is called a $\varphi$-derivation on $A$.

     The author had  already obtained  the proof of the following theorem % and Theorem \ref{Coro.2}
      for
     commutative Banach algebras. After  writing  this paper,
      Professor E. Kaniuth   kindly  brought   the preprint
       \cite{l.k.p} to my attention, in which   the theorem  has been
    similarly and  simultaneously  proved  in the general case. So, the proof of the
     following theorem  is referred to \cite{l.k.p}.

%%%%%%%%%%%%%%%%%%%%%%%%%%%%%%%%%%%%%%%%%%%%%%%%%%%%%%%%%%%%%%%%%%%%%%%%%%

\begin{theorem}\label{Coro.2}
\emph{\cite{l.k.p} Let $A$ be a  Banach algebra   and $\varphi\in Hom(A, \CC)$.
Then  $A$ is $\varphi$-amenable if and only if $A$ is $\varphi$-left amenable. Further,
$A$ has a bounded right approximate identity and   is $\varphi$-amenable if and only if
the maximal ideal $I_\varphi(A)$ has a bounded right approximate identity.
% the
%following   statements are equivalent:}
%\begin{itemize}
%\item [\emph{(i)}] \emph{$A$ is $\varphi$-amenable.}
%\item [\emph{(ii)}] \emph{$A$ is $\varphi$-left amenable.}
%\item[\emph{(iii)}]  \emph{$A^{\ast\ast}$ is $\pi(\varphi)$-amenable.}
%\item[\emph{(iv)}] \emph{ for every $f\in A^{\ast}$, $\lambda\varphi\in \overline{K(f)}^{w^\ast}$  for some $\lambda$.}
%\end{itemize}
}
\end{theorem}

%___________________________________________________________

 %%%%%%%%%%%%%%%%%%%%%%%%%%%%%%%%%%%%%%%%%%%%%%%%%%%%%%%%%
%\begin{theorem}\label{Coro.2}
%\emph{\cite{l.k.p} Let $A$ be a  Banach algebra and  $\varphi\in Hom(A, \CC)$.  $I_\varphi(A)$ has
% a bounded right approximate identity if and only if
% $A$ is $\varphi$-amenable and $A$ has a bounded right approximate identity.}
%\end{theorem}

  It is  of vital importance  to observe that in  Theorem \ref{Coro.2} the boundedness of
                 approximate identities  cannot be omitted either in $A$ or  in  $I_\varphi(A)$; see  Section \ref{e.acc.} (iii).

\begin{corollary}\label{Co.1}
\emph{ By Theorems \ref{Coro.2}  we see that if  $I_\varphi(A)$ has a bounded right approximate identity, then every $X^\ast$-valued continuous
 derivation on $A$ is zero whenever   $X$ is a Banach
 $\varphi$-$A$-bimodule.}
\end{corollary}

%___________________________________________________________________
In various references, for particular cases,   it is already known that  in
 the spectral sets every bounded
point derivation is zero, but we have not seen it in general form.
 It seems worthwhile to give a complete proof of it here.

\begin{proposition}\label{L.7}
\emph{Let $A$ be a    commutative, semisimple and regular Banach algebra and   $\varphi\in
 \Phi(A)$. If $\{\varphi\}$ is a spectral set,  then every derivation from  $A$ into
   $X^\ast$ is zero whenever $X$ is a Banach $\varphi$-$A$-bimodule.}
\end{proposition}

\begin{proof} Let $X$ be a Banach $\varphi$-$A$-bimodule and   $D:A\rightarrow X^\ast$  a
derivation. For $x$, a fixed element of $X$, define
$D_x:A\rightarrow \CC$ by $D_x(a):=D(a)(x)$. Then $D_x$ is a $\varphi$-derivation on $A$.
  The  set  $J:=\{a\in A:
 \;\varphi(a) = D_x(a) = 0\}$ is an ideal
 in $A$ with $Co(J)=\{\varphi\}$, where $Co(J):=\{\phi\in \Phi(A): \phi(a)=0,  \forall a\in J\}$. To see   this,  let $a\in A$
  and $b\in J$. Since
 $D_x(a b)=\varphi(a) D_x(b) + \varphi(b)D_x(a)=0$,
we have    $a b\in J$.  If $b\in I_\varphi(A)$, then
  $D_x(b^2)=2\varphi(b)D_x(b)=0$, which   implies that $b^2\in J$. For every   $\phi\in
Co(J)$ we have     $\phi(b)^2=\phi(b^2)=0$, hence
$I_\varphi(A)\subseteq I_\phi(A)$, and thus  $\varphi=\phi$; thence
$J=I_\varphi(A)$, since $\{\varphi\}$ is a spectral set for $A$. Let
$a\in A$ with $\varphi(a)=1$. Evidently $a^2-a\in I_\varphi(A)$, and
since $D_x|_{I_\varphi(A)}=0$,  $D_x(a^2)=D_x(a)$,  so $D_x(a)=0$
which completes the proof.
 \end{proof}
%_________________________________________________________________

\begin{remark}

 \emph{Observe that in general a Banach algebra  $A$ is
    not necessarily $\varphi$-amenable  if $\{\varphi\}$ is a spectral set for $A$.   For
    example, let $A$ be  the hypergroup algebra of the Bessel-Kingman
    hypergroup of order $\nu=0$. Then $A$ is $\varphi$-amenable if and only if
    $\varphi$ is associated with the identity character on the hypergroup \cite{azimifard.Monath} although every character  of $A$ is a
    spectral set \cite{herz.1}.}

 \end{remark}

%________________________________________________________________
\begin{remark}
     \emph{
    One may pose the     question of whether   $A$ is   amenable if  $A$ is $\varphi$-amenable
    for every $\varphi\in Hom(A, \CC)$
    and if
    functionals $m_\varphi$ associated to
the characters $\varphi$  are uniformly bounded.
The following example from \cite{Bade.dales.curtes} shows that, in general, these
assumptions cannot  imply  even  weak amenability of $A$.}
\end{remark}

 {\bf{Example:}}

Let $AC$ be the set of absolutely continuous functions on the unit
circle $\TT$. For $f\in AC$, set $\left\| f\right\|=\left| f
\right|_{\TT}+\int_0^{2\pi}\left| f'(e^{i\theta})\right|d\theta$,
where $\left| \cdot  \right|_{\TT}$ denotes the uniform norm on
$\TT$. Then $(AC, \left\| \cdot \right\|)$ is a Banach function
algebra on $\TT.$ The Banach algebra  $L^1(\TT)$ is a commutative
Banach $AC$-bimodule with respect to the pointwise product of
functions, and the map $f\mapsto f'$, $AC\rightarrow L^1(\TT)$, is a
nonzero derivation. Thus $AC$ is not weakly amenable
\cite{Bade.dales.curtes}. Let $M$ be a maximal ideal of $AC$, so
that $M=\{f\in AC: f(z_0)=0\}$ for some $z_0\in \TT.$ Set
$e_n(z)=\min \left\{ n\left| z-z_0 \right|, 1 \right\}$. Then it is
easy to check that $\{e_n\}$ is a bounded approximate identity in
$M$, hence,  by Corollary \ref{Co.1},  there are no nonzero
$\varphi_{z_0}$-derivations on $AC$. Let  $m_{z_0}$ be an
accumulation point of $\{e_n\}_n$ in the bidual space
$AC^{\ast\ast}$.
 Let  $u\in AC$ such that  $u(z_0)=1$. Then    the functional $M_{z_0}=u-m_{z_0}u$ is a
  $\varphi_{z_0}$-mean on $AC^\ast$,  and since  $\left|e_n(z)\right|\leq 1$,
  obviously $\left\| m_{z_0} \right\|\leq 2$ when $z_0$ varies on $\TT$.\\

 In the remaining of this section we investigate the $\varphi$-amenable Banach algebras with unique $\varphi$-means.
 To continue we may require the following simple lemma.

\begin{lemma}\label{lemma.unique}
\emph{Let $X$ be a    linear, closed, $A$-left invariant  subspace of $A^\ast$.
If $X$ is $\varphi$-left amenable with a  $\varphi$-mean $m_\varphi$, then
$m_\varphi(\phi)=\delta_\varphi(\phi)$
  for every $\phi\in Hom(A, \CC).$}
\end{lemma}
\begin{proof}
For every   $\phi\in Hom(A, \CC)$ and $a\in A$ we have
\[
\phi(a)m_\varphi(\phi)=m_\varphi(\phi\cdot a)
=\varphi(a)m_\varphi(\phi),\notag
\]
which implies that  $m_\varphi(\phi)=\delta_\varphi(\phi)$.
\end{proof}

%

%__________________________________________________________
Let $X$ be a   linear, closed,  $A$-left invariant subspace of $A^\ast$ such that $X^\ast\cdot  X\subseteq X$.
%In the upcoming paragraphs we study the case that $X$ is  $\varphi$-left amenable  with a unique $\varphi$-mean $m_\varphi$.
%
Moreover, suppose that $X^\ast$ with the following product is a
Banach algebra,
\begin{equation}\label{equ.2}
\langle m\cdot n,f\rangle=\langle m,n\cdot f\rangle,\hspace{1.cm} \mbox{ for every }m,n\in X^\ast \mbox{ and } f\in X.
\end{equation}
The topological centre of $X^{\ast}$  is defined as follows:
 \begin{align}\label{1.14.10}
Z(  X^{\ast}):=\Big{\{} m\in  X^{\ast}: \;\mbox{ the map }& n
\longrightarrow m\cdot  n \;\mbox{ is }\\ \notag &\; \mbox{
$w^\ast$-$w^{\ast}$ continuous  on\;}
 X^{\ast}\Big{\}}. \notag
\end{align}
Note that   $Z(  X^{\ast})$ is a closed  subalgebra of $X^\ast$.

%Observe  that for $G$ fixed in $A^{**}$, the mapping $F\longrightarrow
%F.G$ is always $w^*$-$w^*$ continuous.
%  Moreover   $A \subseteq
%Z(A^{\ast\ast})$ and  $Z(A^{\ast\ast})$ is a closed subalgebra of $
%A^{\ast\ast}$, see \cite{Duncan.Hosseiniun}.

% The following lemma to prove Proposition \ref{T.6}:

\begin{theorem}\label{T.6}
\emph{Suppose  $ X\subseteq A^\ast$ is as above, $Hom(A,
\CC)\subseteq X$,  and $\varphi\in Hom(A, \CC)$. If $X$ is
$\varphi$-left amenable with  a  unique $\varphi$-mean $m_\varphi$,
then $m_\varphi\in Z(X^{\ast})$.  }
\end{theorem}

\begin{proof}
Let $n\in X^{\ast}$, $f\in X$, and $a\in A$. Since $X\cdot  A\subseteq X$ and $X^\ast \cdot  X\subseteq X$, we have
\begin{align}\notag
\langle m_\varphi\cdot n,f\cdot a\rangle & = \langle m_\varphi ,
n\cdot (f\cdot a)\rangle
\\ \notag
 &=\langle m_\varphi, (n\cdot f)\cdot a\rangle\\
\notag &=\varphi(a)\langle m_\varphi, n\cdot f\rangle \\
\notag &=  \varphi(a)\langle m_\varphi\cdot  n , f\rangle. \notag
\end{align}
Since $m_\varphi$ is a unique $\varphi$-mean on $X$, $\langle
m_\varphi,\varphi\rangle =1$  and $ \langle n \cdot  \varphi, a
\rangle =\langle n, \varphi\cdot  a\rangle= \varphi(a)\langle n,
\varphi\rangle $, we have $m_\varphi\cdot  n=\lambda_n n$,
 where $\lambda_n=\langle n, \varphi\rangle .$
Let   $n\in X^{\ast}$  and $\{n_i\}_i$ be a net in
$X^{\ast}$
  such that
 $n_i\overset{w^\ast}{\rightarrow} n$. Then
\[\lambda_{n_i}=\langle n_i,\varphi\rangle \rightarrow
\langle n, \varphi\rangle =\lambda_n, \hspace{.2in}\mbox{ as } i\rightarrow \infty,
\]
hence  for every $f\in A^\ast$ we have
\begin{align}\label{1.56.2}\notag
\langle m_\varphi\cdot n_i,f\rangle\;&=\langle \lambda_im_\varphi,f\rangle\;\\
\notag &=\lambda_i \cdot  \langle m_\varphi, f\rangle\;
 \rightarrow \lambda_n.\langle m_\varphi,f\rangle=\langle m_\varphi\cdot n,f\rangle,
 \hspace{.2in}\mbox{as}
  \;i\rightarrow \infty . \notag
\end{align}
 Thus,  the map  $n\rightarrow {m_\varphi}\cdot  n$ is
$w^*$-$w^*$ continuous on $X^\ast$,
 hence
$m_\varphi\in Z({X^\ast}).$
\end{proof}
%_________________________________________________________________________________________________________________-

\begin{remark}
\emph{
\begin{itemize}
    \item[(i)] In the above theorem, suppose    that  $A$ has a  bounded approximate identity,
      bounded by 1,  $X=A^\ast A$,  and $A Z(A^{\ast\ast})\subseteq A$.
We note  that $X^\ast$ with the product defined in (\ref{equ.2}) is a Banach algebra, and  $Z(X^\ast)$ can be identified
with the right multiplier algebra of $A$, in particular $AZ(X^\ast)\subseteq A$; see \cite[p.1295]{L.U}. So,
if  $X$ is $\varphi$-amenable with a unique $\varphi$-mean
$m_\varphi$, then  the map
 $T_{m_\varphi}:A\rightarrow A$ defined by
    $T_{m_\varphi}(a)=a\cdot m_\varphi$ is a right multiplier on $A$,
    i.e. $T_{m_\varphi}(ab)=aT_{m_\varphi}(b)$,  for every $a,b\in A$.
     In addition, if $A$ is semisimple and commutative, then there exits
    a $\psi\in C^b(\Phi(A))$ such that $\widehat{T_{m_\varphi}a}=\psi \widehat{a}$
    with  $\|\psi\|_\infty\leq \|m_\varphi\|$, where
    $\widehat{a}$ denotes  the Gelfand transform of $a$;
    see \cite[1.2.2]{larsen.Mu.}. Consequently,
    by Lemma \ref{lemma.unique} we have
    $\psi(\phi)=\delta_\varphi(\phi)$ which implies that the character $\varphi$ is  isolated  in $\Phi(A)$.
    \item[(ii)] It is crucial to note that the $\varphi$-left amenability of $X$  with a unique $\varphi$-mean
     depends  on  subspaces $X$ and the characters $\varphi$. For instance, if $G$ is a
     locally compact group, then   $L^\infty(G)$
    is 1-amenable with a unique 1-mean if and only if $G$ is compact; see \cite{Pat88}.
     However, subalgebras $AP(G)$ and $WAP(G)$ of $L^\infty(G)$ are always 1-amenable with the unique 1-mean \cite[p.88]{Pat88}.
    \item[(iii)] Let  $A$ be  $L^1(K)$ when $K$ is the little q-Legendre polynomial hypergroup.
     Then $A$ is $\varphi$-amenable at every character $\varphi\in \Phi(A)$;
     and,
     the cardinality of $\varphi$-means is one  and infinity if $\varphi$ is associated
     with nontrivial and trivial characters of the hypergroup, respectively \cite{azimifard.Monath}. However,
          $WAP(K)$ has only  the unique  $\varphi$-means at every character $\varphi$, \cite{azimifard.Monath, Ska92, Wolf.WAP}.\\
         We may also observe that
          $L^\infty(K)^\ast$ ($K$ is a commutative hypergroup) is isometrically  isomorphic to
          the Banach algebra of finitely additive measures on $K$; see  \cite{stromberg}. So, if  $L^\infty(K)$
          is $\varphi$-amenable with the unique $\varphi$-mean
          $m_\varphi$, then by part (i)
          the finitely additive measure associated with   $m_\varphi$ belongs to the multiplier algebra
          of $L^1(K)$, i.e. $M(K)$ \cite{BloHey94}.
          % hence, it is a countably  additive measure on $K$.
         \\[-1.cm]\end{itemize}
         }

\end{remark}

\section{$\varphi$-Amenability of Direct Sum of Banach Algebras}\label{sub.amenability}

In this section we  study  the $\varphi$-amenability of the Banach
algebra  $A\oplus_{\phi} B$,
 the direct sum of two given Banach algebras $A$ and $B$ where its  algebraic structure
depends on $\phi\in Hom(B, \CC)$. Such  algebras in special cases have  been studied in   \cite{thesis, Lau83}.
% and
%This algebra, for  a certain Banach algebra $B$,  has been already studied by
% A.T.M. Lau in  \cite{Lau83}, and we studied the  case for $L^1$-algebras of commutative hypergroups in \cite{thesis}.
\begin{definition}\label{1.31}
\emph{Let $A$ and $B$ be Banach algebras and $\phi\in Hom(B, \CC)$.
 Define direct sum of $A$ and $B$,
 denoted by
$A\oplus_{\phi} B$, to be the algebra over the complex numbers
consisting of all ordered pairs $(a, b), a\in A,  b\in B$ with
coordinatewise addition and
 scalar multiplication,
and the  product of two elements $a=(a_1,a_2)$ and $b=(b_1,b_2)$ is defined by}
\begin{equation}\label{1.31.1}
a\cdot_{\phi} b =(a_1 b_1+\phi(a_2)b_1+\phi(b_2)a_1, a_2 b_2).
\end{equation}
\end{definition}
%________________________________________________________________________________________
%
%
For every $a=(a_1,a_2)$ as above  define $ \|(a_1,a_2)
\|_{\oplus_{\phi}} = \| a_1 \|+
 \| a_2 \|$.  Plainly  we have  $\|a\cdot_\phi b \|_{\oplus_{\phi}}\leq \|a \|_{\oplus_{\phi}} \|b \|_{\oplus_{\phi}}$,  and so
 $\|\cdot  \|_{\oplus_{\phi}}$
is a norm on $A\oplus_{\phi} B$.
                                If $A$ and $B$ are  Banach $\ast$-algebras,
then $(A\oplus_{\phi} B,\|\cdot \|_{\oplus_{\phi}}) $ with the
the involution  defined by $a^\ast:=(a_1^{\ast}, a_2^{\ast})$ and  the
algebraic product in
 (\ref{1.31.1})   is a  Banach $\ast$ - algebra.
\index{$\|\cdot \|_{\oplus_{\phi}}$}
Obviously $(A, \|\cdot \|)$ and $(B, \|\cdot  \|)$ are subalgebras of
$A\oplus_{\phi} B$.
%
%
%________________________________________________________________________________________
%
The following proposition  is the special case of
%from the general
%properties of direct sum of  normed spaces
\cite[1.10.13]{maginson}.
\begin{proposition}\label{theorem.1.1}
\emph{Let $A$ and $B$ be  Banach algebras. Then $\left(A\oplus_\phi
B\right)^\ast$ is isometrically isomorphic  to $A^\ast \times
B^\ast$ with the norm $\left\| (f,g)\right\| =\mbox{max} \left\{ \|
f\|,\| g\|\right\} $, $f\in A^\ast $ and $g \in B^\ast$. The
isomorphism
\[\eta: A^\ast \times B^\ast \longrightarrow
\left(A\oplus_\phi B\right)^{\ast}\]
  is given by
\begin{equation}\notag
   \langle \eta(f, g), (a,b)\rangle =f(a)+g(b),
\end{equation}
  for every   $(f, g)\in A^\ast\times B^\ast$ and  $(a,b)\in A\oplus_\phi
   B$.}
\end{proposition}

%_________________________________________________________________________________

If $A^{\ast\ast}$ and
$B^{\ast\ast}$ are equipped with the  Arens product, then
$A^{\ast\ast}\oplus_{\phi}B^{\ast\ast}$ is an  algebra as above and
%over the complex numbers  consisting of all ordered pairs $(m_1,
%m_2),\;m_1\in A^{\ast\ast},\;m_2\in B^{\ast\ast}$ with
%coordinatewise addition and
% scalar multiplication,
%and
the  product of two elements $(m_1, m_2)$ and $(m'_1, m'_2)$ of
$A^{\ast\ast}\oplus_{\phi}B^{\ast\ast}$
  is given by
\begin{equation}\label{t.4}\notag
(m_1, m_2) \odot_\phi (m'_1, m'_2) =(m_1\cdot m'_1+m_2(\phi)m'_1+
m'_2(\phi)m_1, m_2\cdot m'_2).
\end{equation}
%
%
%
%______________________________________________________________________________________
\begin{proposition}\label{t.6}
\emph{ Let $A$ and $B$ be Banach algebras. Then
$A^{\ast\ast}\oplus_\phi B^{\ast\ast}$ is isometric and
 algebraically isomorphic to the bidual  algebra
$\left(A\oplus_\phi B\right)^{\ast\ast}$.}
\end{proposition}

\begin{proof}
Define the  map  $J:A^{\ast\ast}\oplus_\phi B^{\ast\ast}\longrightarrow
\left(A\oplus_\phi B\right)^{\ast\ast}$ by  $(m_1, m_2) \mapsto J(m_1,
m_2)$   such that
 \begin{equation}\label{r.r.2}\notag
   \langle J(m_1, m_2),
(f,g)\rangle =m_1(f)+m_2(g),
\end{equation}
 for every $(m_1, m_2)\in
A^{\ast\ast}\times B^{\ast\ast}$ and $(f,g)\in \left(A\oplus_\phi
B\right)^\ast$.
The map $J$ is   a linear       isomorphism
\cite[1.10.13]{maginson}, so we shall now   check that the
isomorphism is  algebraic as well.
For every $(f, g)\in \left(A\oplus B\right)^\ast$ it is easily to
verify that
\begin{equation}\notag
J(m_1, m_2)\cdot(f,g)=\left( m_1\cdot f +  m_2(\phi) f,   m_1(f) \phi + m_2\cdot g\right).
\end{equation}

Let $(m_1, m_2)$, $ (m'_1, m'_2)\in A^{\ast\ast}\times
B^{\ast\ast}$. Then
\begin{align}\notag
\langle J(m_1,m_2)\cdot J(m'_1,&  m'_2), (f,g)\rangle
= \langle J(m_1,m_2),  J(m'_1,m'_2)\cdot (f,g)\rangle \\ \notag
&=\langle J(m_1,m_2), (m'_1\cdot f+ m'_2(\phi)f, m'_1(f)\phi +m'_2\cdot g)\rangle \\ \notag
 &= \langle m_1, m'_1\cdot f+m'_2(\phi)f\rangle+ \langle  m_2,  m'_1(f) \phi+m'_2\cdot g\rangle\\ \notag
 &=\langle m_1,m'_1\cdot f\rangle + m'_2(\phi) m_1(f)+  m'_1(f)m_2 (\phi) +\langle m_2, m'_2\cdot g\rangle \\ \notag
 & =\langle {\left( m_1\cdot m'_1 + m'_2(\phi)m_1+m_2(\phi)m'_1,  m_2\cdot m'_2 \right),
(f,g)}\rangle \\ \notag
 &= \langle J((m_1, m_2) \odot_\phi(m'_1, m'_2)), (f, g)\rangle, \notag
\end{align}
 hence
\[J\left((m_1, m_2)\odot_\phi(m'_1, m'_2)\right)=J(m_1, m_2)\cdot J(m'_1, m'_2).\]
 \end{proof}
%
%_________________________________________________________________________________
%
%
\begin{theorem}\label{t.7}
\emph{ Let $A$ and $B$ be Banach algebras. Then  the following
algebraic isomorphism holds
\begin{equation}\notag
Z\left((A\oplus_\phi B)^{\ast\ast}\right)\cong Z(A^{\ast\ast})\oplus_\phi
Z(B^{\ast\ast}).
\end{equation}
}
\end{theorem}

\begin{proof}
To show $"\subseteq "$, by previous proposition,  we may     let
$(m_1, m_2)\in Z\left(A^{\ast\ast}\oplus_\phi B^{\ast\ast}\right)$.
The map
\[\theta:
A^{\ast\ast}\oplus_\phi B^{\ast\ast} \longrightarrow
A^{\ast\ast}\oplus_\phi B^{\ast\ast},
\]
defined by
 \begin{equation}\label{equation.1}\notag
 (m'_1, m'_2)\mapsto  (m_1, m_2)\odot_\phi (m'_1, m'_2),
\end{equation}
 is $w^\ast$ - $w^\ast$ continuous. Let  now $(m'_i,0)\overset{w^\ast}{\longrightarrow}(m'_1,0)$, as $i\rightarrow \infty$.
Then
\begin{align}\notag
       | \langle (m_1,m_2)   & \oplus_\phi (m'_i,0),  (f,g) \rangle -
                    \langle (m_1,m_2)\oplus_\phi (m'_1,0), (f,g)\rangle
      | \\ \notag
   &=|
            \langle \left(m_1\cdot m'_i+m_2(\phi)m'_i,0\right), (f,g)
            \rangle
             -
                          \langle \left(m_1 \cdot m'_1 +m_2(\phi)m'_1,0 \right), (f,g)
              \rangle
    |   \\ \label{equation.2}
  & =| \langle m_1\cdot m'_i,f \rangle -
          \langle m_1\cdot m'_1, f \rangle
           +m_2(\phi) \left(\langle m'_i,f\rangle -\langle m'_1, f\rangle
                     \right)   |,
\end{align}
 and
\begin{align}\notag
\left|\langle m_1\cdot m'_i,f \rangle-\langle m_1\cdot m'_1,f
\rangle\right|\leq &\left|\langle m_1\cdot m'_i,f \rangle -\langle
m_1\cdot  m'_1,f \rangle +m_2(\phi)\left(\langle m'_i,f\rangle
-\langle m'_1, f\rangle \right)\right|\\ \label{equation.3}
&+\left|m_2(\phi)\right|\left|\langle m'_i,f\rangle -\langle m'_1,
f\rangle \right|.
 \end{align}

 Since   $m'_i\overset{w^\ast}{\longrightarrow}m'_1$ (as $i\rightarrow \infty$) and
the inequality  (\ref{equation.2}) holds,  the  inequality
(\ref{equation.3}) implies that   the map $m'_1\mapsto  m_1\cdot
m'_1$ is $w^\ast$-$w^\ast$  continuous on $A^{\ast\ast}$. Likewise
one can verify  that
% Since
%$\theta$ is $w^\ast$-$w^\ast$ continuous,  clearly
the map $m'_2\mapsto m_2\cdot m'_2$ is also  $w^\ast$-$w^\ast$
continuous on $B^{\ast\ast}$, hence  $(m_1, m_2)\in Z(A^{\ast
\ast})\oplus_\phi Z(B^{\ast\ast})$.

To show $" \supseteq "$,  let $(m_1, m_2)\in Z(A^{\ast
\ast})\oplus_\phi Z(B^{\ast\ast})$ and consider the map
\begin{equation}\label{t.5}
(m'_1, m'_2) \mapsto (m_1, m_2)\odot_\phi (m'_1, m'_2)
\end{equation}
on $A^{\ast\ast}\oplus_\phi B^{\ast\ast}$. Assuming  $(m'_i,
m''_i)\overset{w^\ast}{ \longrightarrow}(m_1,m_2)$, as $i\rightarrow
\infty$,  and $(f,g)\in A^\ast \times B^\ast$ imply  that
\begin{align}\notag
|\langle(m_1,m_2)\odot_\phi (m'_i,m''_i)-& (m_1,m_2)\odot_\phi
(m'_1,m'_2), (f,g)\rangle|
 \leq \\ \notag
 &\left| \langle m_1\cdot m'_i- m_1\cdot m'_1,f \rangle \right|
+\left| m_2(\phi)\right|
 \left| \langle m'_i- m'_1, f \rangle\right|\\ \label{inequality.1}
 &+ \left|  m_1(f)\right|\left| \langle m''_i-m'_2, \phi \rangle \right|+ \left|\langle m_2\cdot m''_i-
m_2\cdot m'_2,g\rangle \right|.
\end{align}
 Since the maps
$m'_1  \mapsto  m_1\cdot  m'_1 $ and $m'_2 \mapsto  m_2\cdot m'_2$
are $w^\ast$-$w^\ast$ continuous on $A^{\ast\ast}$ and
$B^{\ast\ast}$ respectively, the inequality (\ref{inequality.1})
implies that the map (\ref{t.5}) is $w^\ast$-$w^\ast$ continuous on
$A^{\ast\ast}\oplus_\phi B^{\ast\ast}$. Thence,   $(m_1, m_2)\in
Z\left(A^{\ast\ast}\oplus_\phi B^{\ast\ast}\right)$,  and as a
result the identity map  is  an algebraic isomorphism.
% between
%$Z\left((A\oplus_\phi B)^{\ast\ast}\right)$ and
%$Z(A^{\ast\ast})\oplus_\phi Z(B^{\ast\ast})$.
\end{proof}

%_______________________________________________________________________

%The following  proposition   determines the character space of
%$A\oplus_\phi B$.
\begin{proposition}\label{1.33}
\emph{ Let $A$ and $B$ be  Banach algebras. Then   $Hom(A\oplus_{\phi} B, \CC)$ is isomorphic to     $Hom(A, \CC)\times
 \{\phi\}$.}
%\begin{equation}\notag
%T=\Big{\{}{\theta \in (A\oplus_{\phi} B)}^{\ast },\theta (a._\phi
%b)=\theta (a)\theta (b)\;\emph{for all } a \emph{ and } b\in A
%\oplus_{\phi} B\Big{\}}.
%\end{equation}
%\emph{ Then
% $T$ can be identified with $\Phi(A)\times
% \{\phi\}$.}
\end{proposition}

\begin{proof}
 Define $\rho: Hom(A, \CC)\times
\{\phi\}\rightarrow  Hom(A\oplus_{\phi} B, \CC) $ by  $(\psi,\phi) \mapsto \rho_{(\psi,\phi)}$, where
\[\rho_{(\psi,\phi)}(a_1,a_2):=
\psi(a_1)+\phi(a_2),  \hspace{.5in} (a_1,a_2)\in A\oplus_{\phi} B.\]
 The map $\rho_{(\psi,\phi)}$ is
 clearly linear and    bounded, and also multiplicative.  To verify the latter,  let $a=(a_1,a_2)$ and
 $b=(b_1, b_2)$ be in $A\oplus_{\phi} B$. Then
\begin{align}\notag
\rho_{(\psi,\phi)}(a\cdot _\phi b)&=\rho_{(\psi,\phi)} (a_1
b_1+\phi(a_2)b_1+\phi(b_2)a_1,a_2 b_2)\\ \notag &=\psi(a_1b_1)+
\phi(a_2)\psi(b_1)+\phi(b_2)\psi(a_1)+\phi(a_2b_2)\\ \notag
&=\psi(a_1)\psi(b_1)+\phi(a_2)\psi(b_1)+\phi(b_2)\psi(a_1)+\phi(a_2)\phi(b_2)
 \\ \notag
&=\rho_{(\psi,\phi)}(a)\cdot \rho_{(\psi,\phi)}(b).
\end{align}

In order to prove  the converse of the theorem,   let  $\theta \in
Hom(A\oplus_{\phi} B, \CC) $. Obviously $\theta|_{B}={\phi'}$ for
some $\phi'\in Hom(B, \CC)$
 and  there  exists  a $\psi \in Hom(A, \CC)$
  such that $\theta \mid _{A}=\psi$; thus, $\theta(a_1,a_2)=\psi(a_1)+\phi'(a_2)$. Let $a_1\in A$ such that $\psi(a_1)\not=0$.
  Then $\theta((a_1,a_2)\cdot (0, b_2))=\phi(b_2)\psi(a_1)+\phi'(a_2)\phi'(b_2)$, on the other hand, since
  $\theta$ is multiplicative, we have $\theta((a_1,a_2)\cdot (0, b_2))=\psi(a_1)\phi'(b_2)+\phi'(a_2)\phi'(b_2)$,
  which implies that $\phi(b_2)=\phi'(b_2)$,  for every  $b_2\in B$; thence   $\theta
=\rho_{(\psi,\phi)}$.
\end{proof}

%_______________________________________________________________________________

\begin{remark}
\emph{
  Let  $\{e_i\}_i$ be   a  bounded left approximate identity   for $B$
such that $\underset{i\rightarrow \infty}{\lim}\;\phi({e_i})=1$. For
every $(a, b)\in A\oplus_\phi B$ we have
\[\|(a, b)-(0, e_i)\cdot(a, b)\|_{\oplus_\phi} = \|\left(\phi({e_i})-1\right)a\|+ \|e_i b-b\|, \]
hence   $\{(0, e_i)\}$ is a  bounded left approximate identity  for
$A\oplus_\phi B$.}
\end{remark}

%_________________________________________________________________________________________________

\begin{theorem}\label{1.34}
\emph{  Let $A$ and $B$ be  Banach algebras and $\varphi
\in
 Hom(A, \CC)$.  Then $A\oplus_{\phi} B$ is
$\rho_{(\varphi,\phi)}$-amenable
 if and only if $ A$ is $\varphi$-amenable.}
\end{theorem}
\begin{proof}
Let $A$ be $\varphi$-amenable.
  Then there exists a $m_\varphi\in
A^{\ast\ast}$  such that $m_{\varphi}(\varphi)=1$ and
$m_\varphi(f\cdot a)=\varphi(a)m_\varphi(f)$ for every $f\in A^\ast$
and $a\in A$.
   For any pairs  $(f,g)\in A^{\ast} \times  B^{\ast}$
and $(a,b)\in A\oplus_{\phi} B$, $(f,g)\cdot(a,b)=(f\cdot a+\phi(b)
f,f(a)\phi+g\cdot b)$, and we suppose $\rho_{\varphi,\phi}$ is as in
the preceding proposition. Then
%then by the same way 4.5 in ~\cite{lau}
%   $(f, g)$ and  $(a, b)$ chosen as above  we have
\begin{align}\label{1.34.1}\notag
\langle(m_{\varphi}, 0), (f, g)\cdot(a,b)\rangle &=\langle(m_{\varphi},
0), (f\cdot a+ \phi(b)f, f(a) \phi+g\cdot b)\rangle \\ \notag
 &=\langle m_\varphi,
f\cdot a+\phi(b)f\rangle\\ \notag &=\varphi(a)\langle m_\varphi
,f\rangle + \phi(b) \langle m_\varphi, f \rangle \\ \notag
&=\rho_{(\varphi,\phi)}(a,b)\;m_\varphi(f)\\ \notag
&=\rho_{(\varphi,\phi)}(a,b)\langle (m_\varphi,0), (f,g)\rangle,
\end{align}
and $\langle (m_\varphi,0),\rho_{\varphi,\phi}\rangle=1$; hence,
$A\oplus_{\phi} B$ is $\rho_{(\varphi,\phi)}$-amenable.

To show  the converse,
%let $A\oplus_{\phi} B$ be
%$\rho_{\varphi,\phi}$-amenable and
assume  $(m'_\varphi,m'_\phi)$ is a $\rho_{\varphi,\phi}$-mean on
$\left(A\oplus_{\phi} B\right)^\ast$.
 We shall check that $\langle
m'_\varphi,\varphi\rangle \not=0$.
Since  $(m'_\varphi, m'_\phi)$ is a
$\rho_{\varphi,\phi}$-mean,
$$\langle (m'_\varphi, m'_\phi),
(f,0)\cdot (a,0)\rangle=\varphi(a)m'_\varphi(f),$$ hence
\begin{equation}\label{r.r.r.2}
m'_\varphi(f\cdot a)+f(a )m'_\phi(\phi)=\varphi(a)m'_\varphi(f).
\end{equation}
%
%Therefore for  $f=\varphi$ we have
%\begin{equation}\label{r.r.r.1}
%\varphi(a)m'_\varphi(\varphi)+\varphi(a)
%m'_\phi(\phi)=\varphi(a)m'_\varphi(\varphi).
%\end{equation}
%

Let  $a\in A$ with   $\varphi(a)\not=0$. For $f=\varphi$  the equality (\ref{r.r.r.2}) yields  $m'_\phi(\phi)=0$.
As a result, $ m'_\varphi(f\cdot a) =\varphi(a)m'_\varphi(f)$,
 for every  $f\in A^\ast$ and $a\in A$.   Since   $(m'_\varphi, m'_\phi)(\rho_{\varphi,\phi})=1$,  we may have
  $m'_\varphi(\varphi)=1$ which completes the proof.
 \end{proof}

 We note  that,  according to the   previous  theorem,     the character amenability of $A\oplus_{\phi} B$
 is independent of any amenability condition on  $B$.
\\[-1.cm]

%\begin{comment}
%In the following theorem we consider the
%$\varphi$-amenability of projection tensor product of two
%commutative Banach algebras.
%
%\begin{theorem}\label{T.5}
%\emph{Let $A$ and $B$ be  commutative Banach algebras with bounded approximate identity,
% $\varphi\in Hom(A, \CC)$ and $\psi\in Hom(B, \CC)$.
%If $A$ and $B$ are   $\varphi$ and $\psi$ - amenable  respectively,  then
%$A\otimes_p B$ is $\varphi \times \psi$ - amenable, where $A\otimes_p
%B$ denotes the projection tensor product of $A$ and $B$.}
%\end{theorem}
%\begin{proof}
%The character space  $Hom(A\otimes_p B, \CC)$ is homeomorphic to $Hom(A, \CC)\times
%Hom(B, \CC)$;  see \cite{BonDun73}. Let $\varphi\times \psi$ be the
%corresponding character to $\varphi$ and $\psi$.  Since  $A$ and $B$ are $\varphi$ and $\psi$ - amenable respectively,   by
%Theorem
%\ref{Coro.2}  $I_\varphi(A)$ and $I_\psi(B)$  have bounded approximate identity. In that
% $I_{\varphi\times \psi}(A\otimes_p B)=I_\varphi(A)\otimes_p B+ A\otimes_p I_\psi(B)$ has a
%  bounded  approximate identity, \cite[Theorem 2.9.21]{Dales},
  %
%     equivalently by Theorem \ref{Coro.2}  $A\otimes_p B$
%is  ${\varphi\times \psi}$ - amenable.
%\end{proof}
%\end{comment}

 \section{Examples}\label{e.acc.}
 Here  are examples of
   Banach function  algebras for
  which their $\varphi$-amenability relies   on their  characters.
\begin{itemize}
  \item [(i)] Let $A$ be a unital uniform algebra on
a nonempty compact space $\Omega$.  Suppose that $K_A=\{\lambda
\in A': \|\lambda\|=\langle 1,\lambda\rangle =1\}$ and $\Gamma_0(A)$
is
 the set of extremal points of $K_A$.  Then  \cite[4.3.5]{Dales} and   Theorem \ref{Coro.2}
 show that      $x\in \Gamma_0(A)$ if and only if   $A$ is $\varphi_x$-amenable, where $\varphi_x\in \Phi(A)\cong \Omega$.

\item [(ii)] Let $A(\overline{\DD})=\left\{f\in C(\overline{\DD}):
f|_{\DD} \text{ is analytic} \right\}$,
  where $\DD=\left\{z\in \CC: | z| <1\right\}$
is the open unit disc. Then $(A(\overline{\DD}), |\cdot
|_{\overline{\DD}})$
  is a commutative
 Banach algebra where $|f|_{\overline{\DD}}=\sup\{|f(z)|: z\in \overline{\DD}\}$,
 which is called the disc algebra. Since  $\Gamma_0(A(\overline{\DD}))=\TT$, by part
 (i),
 %  see \cite[p.454]{Dales}, hence by Theorem \ref{Coro.2}
$A(\overline{\DD})$ is $\varphi_z$-amenable if and only if  $z\in
\TT$. Observe  that for every $z\in \DD$    there exists a nonzero
continuous point derivation at $z$ on $A(\overline{\DD})$ of  the
form $f\mapsto \alpha f'(z)$ \cite[4.3.13]{Dales}, for some $\alpha\in \CC$; hence,
 $A(\overline{\DD})$ is not weakly amenable \cite[1.5]{Bade.dales.curtes}. %4.3.13 page 454 Dales

 \item [(iii)] Let $K$ be a commutative compact hypergroup. Then
$(C(K), \|\cdot \|_\infty)$ is a proper Segal algebra in  $L^1(K)$.
 Since every maximal ideal in  $L^1(K)$ has a bounded  approximate identity, \cite{thesis},  the maximal ideals in  $C(K)$ have only
 unbounded    approximate identities  \cite{BonDun73}.
 However,  the functional
$m_\varphi:=\pi(\varphi)/\|\varphi\|_2^2$ is the  unique
$\varphi$-mean on $C(K)^\ast$, for every $\varphi\in\Phi(C(K))$.

\end{itemize}

%_____________________________________________________________________________________________

% azimifard@hotmail.de
%\begin{center}
%\small{
%Address: Dietlinden Strasse 16\\
%80802 M\"unchen\\
%Deutschland\\
 %  azimifard@hotmail.de}
 % \end{center}

\end{document}